\documentclass[reqno]{amsart}
\usepackage{inputenc}[utf8]
\usepackage{amsmath, amssymb, amsthm, epsfig}
\usepackage{latexsym}
\usepackage[bookmarksnumbered, hidelinks, colorlinks=false]{hyperref}
\usepackage[noabbrev, capitalize]{cleveref}
\usepackage{xurl}
\usepackage[mathscr]{euscript}
\usepackage{mathtools}
\usepackage[style=trad-alpha]{biblatex}
\bibliography{bib.bib}

\usepackage{xcolor}

\usepackage{fullpage} 
\usepackage{setspace}
\usepackage{caption}
\usepackage{soul}
\usepackage{verbatim}

\title{Sharp Fourier extension for functions with localized support on the circle}
\author[Becker]{Lars Becker}
\address{Mathematical Institute, 
	University of Bonn,
	Endenicher Allee 60, 53115, Bonn,
	Germany}
\email{becker@math.uni-bonn.de}
\date{\today}
\subjclass[2020]{42B10}
\keywords{Circle, Fourier restriction, sharp inequalities}

\theoremstyle{plain}
\newtheorem{theorem}{Theorem}
\newtheorem{lemma}[theorem]{Lemma}
\newtheorem{corollary}[theorem]{Corollary}

\newtheorem{conjecture}[theorem]{Conjecture}
\theoremstyle{definition}

\theoremstyle{remark}

\newcommand{\R}{\mathbb{R}}

\newcommand{\Z}{\mathbb{Z}}
\DeclareMathOperator{\supp}{supp}

\begin{document}

\begin{abstract}
    A well known conjecture states that constant functions are extremizers of the  $L^2 \to L^6$ Tomas-Stein extension inequality for the circle. We prove that functions supported in a $\sqrt{6}/80$-neighbourhood of a pair of antipodal points on $S^1$ satisfy the conjectured sharp inequality. In the process, we make progress on a programm formulated in \cite{Carneiro+2017} to prove the sharp inequality for all functions.
\end{abstract}

\maketitle

\section{Introduction}
We are interested in the conjecture that constant functions are extremizers for the Tomas-Stein Fourier extension inequality for the circle
\begin{equation}
    \label{main ineq}
    \|\widehat{f \sigma}\|_{L^{6}(\R^2)} \leq C \|f\|_{L^2(\sigma)}\,.
\end{equation}
Here $\sigma$ is the arc length measure on the unit circle $S^1 \subset \R^2$ and $\hat \mu(\xi) = \int e^{-ix\xi} \, \mathrm{d}\mu(\xi)$ is the Fourier transform. 

The corresponding conjecture for $S^2$ was proven by Foschi \cite{Foschi2015}, and in \cite{Carneiro+2017} Foschi's argument is adapted to $S^1$, and the conjecture of interest is reduced to the following.

\begin{conjecture}
\label{conjecture}
The quadratic form 
$$
    Q(f) := \int_{(S^1)^6}  (\lvert\omega_1 + \omega_2 + \omega_3\rvert^2 - 1) (f(\omega_1, \omega_2, \omega_3)^2 - f(\omega_1, \omega_2, \omega_3) f(\omega_4, \omega_5, \omega_6)) \,\mathrm{d}\Sigma(\omega)
$$
is positive semi-definite on the subspace $V$ of all antipodal functions in $L^2((S^1)^3, \R)$. Here we denote $\mathrm{d}\Sigma(\omega) = \delta(\sum_{j = 1}^6 \omega_k) \prod_{j = 1}^6 \mathrm{d}\sigma(\omega_j)$, and a function $f$ is antipodal if $f(\pm \omega_1, \pm \omega_2, \pm \omega_3)$ does not depend on the choice of signs.
\end{conjecture}

\Cref{conjecture} has been verified for all functions with Fourier modes up to degree $120$ in \cite{OeS+2022} and \cite{Barker+2020}, via a numerical computation of the eigenvalues of $Q$ on the finite dimensional space of such functions. Further, using different methods, in \cite{Ciccone+2022} the conjectured sharp form of inequality \eqref{main ineq} has been established for certain infinite dimensional subspaces of $L^2(\sigma)$ with constrained Fourier support. Our main result establishes \cref{conjecture} for functions with localized spatial support.

Let $C$ be the cylinder of radius $\varepsilon$ centered at the line $\R (1,1,1)$, and define
$$
    V_\varepsilon := \bigg\{f \in V \, : \, \supp f(e^{i\theta_1}, e^{i\theta_2}, e^{i\theta_3}) \subset \bigcup_{k \in \pi \mathbb{Z}^3} k + C\bigg\}.
$$

\begin{theorem}
    \label{thm main}
    Let $\varepsilon = 1/20$. Then for all $f \in V_{\varepsilon}$ it holds that $Q(f) \geq 0$.
\end{theorem}

Note that since constant functions are in the kernel of $Q$, the same result holds for $V_\varepsilon \oplus \langle \mathbf{1} \rangle$, where $\mathbf{1}$ is the constant $1$ function.

As a corollary, functions with support sufficiently close to a pair of antipodal points satisfy \eqref{main ineq} with the conjectured sharp 
constant.  Define $\Phi(g) := \|\widehat{g\sigma}\|_{L^6(\R^2)} / \|g\|_{L^2(\sigma)}$.

\begin{corollary}
    \label{cor main}
    Let $\varepsilon' = \sqrt{3/8}\varepsilon$.
    Suppose that $g \in L^2(\sigma)$ is such that  $g(e^{i\theta})$ is supported in $( - \varepsilon', \varepsilon') + \pi \mathbb{Z}$. Then $\Phi(g) \leq \Phi(\mathbf{1})$, where $\mathbf{1}$ is the constant $1$ function on $S^1$.
\end{corollary}

Note that by rotation symmetry, the same holds when $g(e^{i\theta})$ is supported in $I + \pi \mathbb{Z}$ for any interval $I$ of length $2\varepsilon'$.

The constants $\varepsilon$ and $\varepsilon'$ in \cref{thm main} and \cref{cor main} are not optimal. Numerical computations suggest that with our method $\varepsilon$ can be improved up to about $0.104$ and $\varepsilon'$ up to about $0.063$, see \cref{sec discussion}. 

The numerical results in \cite{Barker+2020} suggest that eigenfunctions of $Q$ on the subspace of functions with Fourier modes up to degree $N$ corresponding to small eigenvalues concentrate in space. \Cref{thm main} shows that $Q$ is positive on all such sufficiently concentrated functions, thus it should be a useful partial result in establishing positive semi-definiteness of $Q$ on the full space of antipodal functions.
A more precise observation by Jiaxi Cheng, a graduate student in Bonn, is that the smallest eigenvalue is of size $\sim N^{-2}\log(N)$, see Section 2 of \cite{Negro+2022}. The existence of such an eigenvalue is also explained by the asymptotic formula for the multiplier $m$ in \cref{lem I est}, which looks like $c\lvert \log\lvert x\rvert\rvert \lvert x\rvert^2$ near $0$. Unfortunately, we cannot prove that this is the smallest eigenvalue.

More generally, the topic of sharp Fourier extension inequalities has attracted a lot of interest in recent years. In the following we consider general dimensions $d\geq 2$.  Then the Tomas-Stein extension inequality states that for every $q \geq q_d := 2(d+1)/(d-1)$, there exists $C(d,q) > 0$ such that for all $f \in L^2(S^{d-1}, \sigma^{d-1})$
\begin{equation}
    \label{eq gen TomasStein}
    \|\widehat{f\sigma}\|_{L^q(\R^d)} \leq C(d, q) \|f\|_{L^2(\sigma)}\,.
\end{equation}
Here $\sigma^{d-1}$ denotes the $d-1$-dimensional Hausdorff measure on $S^{d-1}$.

It is known that extremizers for \eqref{eq gen TomasStein} exist when $q > q_d$, for all $d$, see \cite{Fanelli+2010}. At the endpoint $q = q_d$, existence and smoothness of extremizers have been shown for $d = 3$ in \cite{Christ+2012}, \cite{Christ+2012b} and for $d = 2$ in \cite{Shao2016}, \cite{Shao2017}. For higher dimensional spheres $d \geq 4$, existence of extremizers for $q = q_d$ is known conditional on the conjecture that Gaussians maximize the corresponding extension inequality for the paraboloid, see \cite{Frank+2016}.

For certain specific choices of $(d,q)$, a full characterization of the extremizers of \eqref{eq gen TomasStein} is known. Most such results grew out of the work of Foschi \cite{Foschi2015}, who showed that constant functions maximize \eqref{eq gen TomasStein} for $(d,q) = (2, 4)$, and gave a full characterization of all complex valued maximizers. His method can be adapted for some non-endpoint extension inequalities on higher dimensional spheres, see \cite{Car+2015}. Using different methods, maximizers of \eqref{eq gen TomasStein} for some choices of $(d,q)$ with even $q > 4$ are characterized in \cite{OeS+2021}.
In some further cases it is known that constant functions are local maximizers. This was shown in \cite{Carneiro+2017} for $(d, q)= (2,6)$, and in \cite{Goncalves+2022} for  $(d, q_d)$ with $3 \leq d \leq 60$. 
For further background and references on sharp Fourier extension inequalities we refer to \cite{Foschi+2017} and \cite{Negro+2022}.

\subsection*{Acknowledgement}
I am grateful to Christoph Thiele for introducing me to this problem and for many helpful discussions, and to Jan Holstermann for pointing out the short proof of \cref{auxilliary4}.
The author was supported by the Collaborative Research Center 1060 funded by the Deutsche
Forschungsgemeinschaft (DFG, German Research Foundation) and the Hausdorff Center for
Mathematics, funded by the DFG under Germany’s Excellence Strategy - GZ 2047/1, ProjectID 390685813.

\section{Proof of \texorpdfstring{\cref{cor main}}{Corollary 3}}
\Cref{cor main} is a direct consequence of \cref{thm main} and the program formulated in \cite{Carneiro+2017}. We give a brief sketch of the implication here, for the details of the program and proofs we refer the reader to \cite{Carneiro+2017}.

\begin{proof}[Proof of \cref{cor main}]
    Let $g \in L^2(\sigma)$ such that $g(e^{i\theta})$ is supported in $(-\varepsilon', \varepsilon') + \pi \mathbb{Z}$. Define $\tilde g(x) = g(-x)$ and 
    $$
        g_\# = \sqrt{\frac{\lvert g\rvert^2 + \lvert \tilde g\rvert^2}{2}}\,.
    $$
    As shown in \cite{Carneiro+2017}, Step 1 and 2, it holds that $\Phi(g) \leq \Phi(g_\#)$,
    and $g_\#$ is antipodal and $g_\#(e^{i\theta})$ is supported in $(-\varepsilon_2, \varepsilon_2) + \pi \mathbb{Z}$.
    Define $f(\omega_1, \omega_2, \omega_3) := g_\#(\omega_1)g_\#(\omega_2)g_\#(\omega_3)$. Then 
    $f \in V_{\sqrt{8/3}\varepsilon'} = V_{\varepsilon}$, 
    hence $Q(f) \geq 0$, by \cref{thm main}. This verifies Conjecture 1.4 in \cite{Carneiro+2017} for $g_\#$. Using Step 3, 4 and 5 in \cite{Carneiro+2017}, we conclude that $\Phi(g) \leq \Phi(1)$.
\end{proof}

As written, this proof depends on Theorem 1.2 in \cite{Carneiro+2017}, which is the main result of that paper and verifies Step 5. However,  this step has a shorter proof if one assumes that the functions have small support, as we do. 
The step consists in proving the inequality
\begin{multline}
    \int h(\omega_1) h(\omega_2) h(\omega_3) (\lvert\omega_1 + \omega_2 + \omega_3\rvert^2 - 1) \, \sigma * \sigma * \sigma(\omega_1 + \omega_2 + \omega_3) \, \mathrm{d}\sigma(\omega_1) \, \mathrm{d}\sigma(\omega_2) \, \mathrm{d}\sigma(\omega_3) \\
    \leq \frac{\|h\|_{L^1(\sigma)}^3}{\|\mathbf{1}\|_{L^1(\sigma)}^3} \int  (\lvert\omega_1 + \omega_2 + \omega_3\rvert^2 - 1) \, \sigma * \sigma * \sigma(\omega_1 + \omega_2 + \omega_3) \, \mathrm{d}\sigma(\omega_1) \, \mathrm{d}\sigma(\omega_2) \, \mathrm{d}\sigma(\omega_3)\,, \label{eq step5}
\end{multline}
valid for all antipodal $h \in L^1(\sigma)$, which is then applied to $h = g^2$. Under the assumptions of \cref{cor main}, $g$ and therefore also $h$ are supported in $( -\varepsilon', \varepsilon') + \pi \mathbb{Z}$. Then \eqref{eq step5} follows from the estimate 
\begin{multline*}
    \max_{[-\varepsilon', \varepsilon']^3} \frac{1}{8} \sum_{\gamma \in \{-1,1\}^3} (\lvert\gamma_1 e^{i\theta_1} + \gamma_2 e^{i\theta_2} + \gamma_3 e^{i\theta_3}\rvert^2 - 1)\, \sigma * \sigma * \sigma(\gamma_1 e^{i\theta_1} + \gamma_2 e^{i\theta_2} + \gamma_3 e^{i\theta_3}) \\
    \leq \frac{1}{\|\mathbf{1}\|_{L^1(\sigma)}^3} \int  (\lvert \omega_1 + \omega_2 + \omega_3\rvert^2 - 1) \, \sigma * \sigma * \sigma(\omega_1 + \omega_2 + \omega_3) \, \mathrm{d}\sigma(\omega_1) \, \mathrm{d}\sigma(\omega_2) \, \mathrm{d}\sigma(\omega_3)\,,
\end{multline*}
which can be verified using \cref{AsymptoticsRho}.

\section{Proof of \texorpdfstring{\cref{thm main}}{Theorem 2}}

\subsection{Orthogonal decomposition}
We consider the sesquilinear form 
$$
    B(f,g) = \int_{(S^1)^6} (\lvert\omega_1+ \omega_2 + \omega_3\rvert^2 - 1) (f(\omega_1, \omega_2, \omega_3)\overline{g(\omega_1,\omega_2,\omega_3)}- f(\omega_1, \omega_2, \omega_3) \overline{g(\omega_4, \omega_5,\omega_6)}) \, \mathrm{d}\Sigma(\omega)\,.
$$
By a change of variables, it holds that $B(f,g) = B(R f, R g)$, where $R f(\omega_1 ,\omega_2, \omega_3) = f(e^{i} \omega_1, e^{i} \omega_2, e^{i} \omega_3)$. Define $Z_d = \{(k_1, k_2, k_3) \in (2\mathbb{Z})^3 \, : \, k_1 + k_2 + k_3 = d\}$ and 
$$
    X_d = \{\sum_{k \in Z_d} a_k \omega_1^{k_1} \omega_2^{k_2} \omega_3^{k_3}  \ : \ a_k \in \ell^2(Z_d)\} \subset L^2((S^1)^3)\,.
$$
For $d \neq d'$, the spaces $X_d$ and $X_{d'}$ are eigenspaces of $R$ with different eigenvalues, and hence are orthogonal with respect to $B$. 
Note that the orthogonal projection $\pi_d$ onto $X_d$ can be expressed as 
$$
    \pi_d(f)(\omega_1, \omega_2, \omega_3) = \int_0^{1} e^{-2\pi idt} f(e^{2\pi it} \omega_1, e^{2\pi it} \omega_2, e^{2 \pi it}\omega_3) \, \mathrm{d}t\,,
$$
which implies that $\pi_d(V_\varepsilon) \subset V_\varepsilon$. Therefore, we have that
$$
    V_\varepsilon = \overline{\bigoplus_{d \in \mathbb{Z}} \pi_d(V_\varepsilon)} = \overline{\bigoplus_{d \in \mathbb{Z}} (V_\varepsilon \cap X_d)}\,.
$$
Hence, it suffices to show positive semi-definiteness of $B$ on each of the spaces $X_{d, \varepsilon} := V_\varepsilon \cap X_d$.

\subsection{Reducing the dimension}
\label{RedDim}
We now integrate out simultaneous rotations of all $\omega_i$ by the same angle. From now on, we use the convention that 
\begin{equation}
    \label{conv1}
    \omega_i = (\cos(\theta_i), \sin(\theta_i))\,, 
\end{equation}
and abuse notation by writing $f(\omega(\theta)) = f(\theta)$. We also define
\begin{align*}
    a(\theta_1, \theta_2, \theta_3) &:=(\cos(\theta_1) + \cos(\theta_2) + \cos(\theta_3))^2 + (\sin(\theta_1) + \sin(\theta_2) + \sin(\theta_3))^2= \lvert \omega_1 + \omega_2 + \omega_3\rvert^2\,,
\end{align*}
so that the weight in the bilinear form $B$ is given by $a - 1$, and record the useful identity
\begin{equation}
    \label{eq a cos}
    a(\theta_1, \theta_2, \theta_3) = 3 + 2 \cos(\theta_1 - \theta_2) + 2 \cos(\theta_2 - \theta_3)+ 2 \cos(\theta_3 - \theta_1)\,.
\end{equation}

\begin{lemma}
    \label{FundamentalDomain}
    Let $C$ be the convex hull of the four points 
    $$
        (\pi,-\pi,0)\,, \quad (-\frac{\pi}{3}, -\frac{\pi}{3}, \frac{2\pi}{3})\,,\quad (-\pi, \pi,0)\,,\quad (\frac{\pi}{3}, \frac{\pi}{3}, -\frac{2\pi}{3})\,.
    $$
    Then the prism $P$ over $C$ of height $2 \pi \sqrt{3}$ is a fundamental domain for  $\mathbb{R}^3 / (2\pi\mathbb{Z})^3$.
\end{lemma}

\begin{proof}
    For each point in $\R^3/ (2\pi \Z)^3$, there exists a unique representative $(\theta_1, \theta_2, \theta_3) \in \R^3$ such that $0 \leq \theta_1+  \theta_2 + \theta_3 < 2\pi$, and such that the orthogonal projection of $(\theta_1, \theta_2, \theta_3)$ onto the hyperplane $\{(x_1,x_2, x_3) \ : \ x_1 + x_2 + x_3 = 0\}$ is contained in a fixed fundamental domain $C_1$ of the lattice generated by $(2\pi, -2\pi,0)$ and $(-2\pi, 0,2\pi)$, such as
    \begin{align}
        C_1 :=
        \left\{
            \lambda_1
            \begin{pmatrix}
                2\pi\\
                -2\pi\\
                0
            \end{pmatrix}
            +
            \lambda_2
            \begin{pmatrix}
                -2\pi\\
                0\\
                2\pi
            \end{pmatrix}
            \ : \ \lambda_1, \lambda_2 \in [0,1)
        \right\}.
    \end{align}
    Thus $C_1 + \{(t,t,t) \, : \, t \in [0, \frac{2\pi}{3})\}$ is a fundamental domain for $\R^3/(2 \pi \mathbb{Z})^3$.
    By decomposing this fundamental domain into finitely many pieces and rearranging them, it is easy to see that $C + \{(t,t,t) \ : \ t \in [0, 2\pi)\}$ is also a fundamental domain.
\end{proof}

\begin{lemma}
    \label{lem dimred}
    For all $d \in \mathbb{Z}$ there exists a function $\lambda_d: C\times C \to S^1$ such that for all $f \in X_d$ 
    $$
        B(f, f) = 12 \pi \int_{C^2} \delta(a(\theta) - a(\theta')) (a(\theta) - 1)  (\lvert f(\theta)\rvert^2 - \lambda_d(\theta', \theta) f(\theta) \overline{f(\theta')}) \, \mathrm{d}\mathcal{H}^2_C(\theta) \, \mathrm{d}\mathcal{H}^2_C(\theta')\,.
    $$
    Here $\mathcal{H}^2_C$ denotes the $2$-dimensional Hausdorff measure on $C$.
\end{lemma}

\begin{proof}
    By Lemma \ref{FundamentalDomain}, we have
    \begin{multline*}
        B(f,f) = \int_{P \times P} \delta(\omega_1 + \omega_2 + \omega_3 - \omega_4 - \omega_5 - \omega_6) (\lvert \omega_1 + \omega_2 + \omega_3\rvert^2 - 1)\\
        \times (\lvert f(\omega_1, \omega_2, \omega_3)\rvert^2- f(\omega_1, \omega_2, \omega_3) \overline{f(\omega_4, \omega_5,\omega_6)}) \prod_{j = 1}^6 \mathrm{d}\theta_j\
    \end{multline*}
    \begin{multline*}
        = 2\pi \sqrt{3} \int_{C \times P} \delta(\omega_1 + \omega_2 + \omega_3 - \omega_4 - \omega_5 - \omega_6) (\lvert \omega_1 + \omega_2 + \omega_3\rvert ^2 - 1)\\
        \times (\lvert f(\omega_1, \omega_2, \omega_3)\rvert ^2- f(\omega_1, \omega_2, \omega_3) \overline{f(\omega_4, \omega_5,\omega_6)}) \,\mathrm{d}\mathcal{H}^2_C(\theta_1, \theta_2, \theta_3) \, \prod_{j = 4}^6 \mathrm{d}\theta_j\,.
    \end{multline*}
    Here we have used that $f \in X_d$, to integrate out simultaneous rotations of all $6$ points $\omega_j$ by the same angle. For $x,y \in \R^2$, it holds that 
    $$
        \delta(x - y)  = 2\delta(\lvert x\rvert^2 - \lvert y\rvert^2) \delta(\arg(x) - \arg(y))\,.
    $$
    Hence, we can rewrite the last expression as
    \begin{multline*}
        = 4\pi \sqrt{3}  \int_{C \times P}  \delta(\lvert \omega_1 + \omega_2 + \omega_3\rvert^2 - \lvert \omega_4 +\omega_5 + \omega_6\rvert^2) \delta( \arg(\omega_1 + \omega_2 + \omega_3) - \arg(\omega_4 + \omega_5 + \omega_6)) \\
        \times(\lvert \omega_1 + \omega_2 + \omega_3\rvert ^2 - 1) (\lvert f(\omega_1, \omega_2, \omega_3)\rvert^2 - f(\omega_1, \omega_2, \omega_3) \overline{\omega_4, \omega_5, \omega_6)})\,
         \mathrm{d}\mathcal{H}^2_C(\theta_1, \theta_2, \theta_3) \, \prod_{j = 4}^6 \mathrm{d}\theta_j
    \end{multline*}
    \begin{multline*}
        = 12 \pi  \int_{C \times C} \int_0^{2 \pi} \delta(a(\theta_1, \theta_2, \theta_3) - a(\theta_4,\theta_5,\theta_6)) \delta( \arg(\omega_1 + \omega_2 + \omega_3) - \arg(\omega_4 + \omega_5 + \omega_6) - t)\\
        \times(a(\theta_1, \theta_2, \theta_3) - 1) (\lvert f(\omega_1, \omega_2, \omega_3)\rvert^2 - f(\omega_1, \omega_2, \omega_3) \overline{f(e^{it}\omega_4, e^{it}\omega_5, e^{it}\omega_6)})\,\mathrm{d}t\, \mathrm{d}\mathcal{H}^4_{C\times C}(\theta) \, \,.
    \end{multline*}
    Since $f \in X_d$, we have $f(e^{it}\omega_4, e^{it} \omega_5, e^{it}\omega_6) = e^{itd} f(\omega_4, \omega_5, \omega_6)$. Thus, we can integrate out $t$ and obtain the claimed identity.
\end{proof}

\subsection{Completing the proof}
By \cref{lem dimred}, we have for all $d$ and all $f \in X_d$:
\begin{align}
    B(f,f) &\geq 12\pi \int_{C} \lvert f(\theta)\rvert^2 (a(\theta) - 1) \int_C \delta(a(\theta) - a(\theta'))   \, \mathrm{d}\mathcal{H}^2_C(\theta')\,
    \mathrm{d}\mathcal{H}^2_C(\theta) \nonumber\\
    &\quad- 12\pi \int_{C^2} \delta(a(\theta) - a(\theta')) \lvert a(\theta) - 1\rvert \lvert f(\theta)\rvert \lvert f(\theta')\rvert\, \mathrm{d}\mathcal{H}^2_C(\theta)\, \mathrm{d}\mathcal{H}^2_C(\theta')\nonumber\\
    &=:12\pi( I - I\!I)\,. \label{eq I def}
\end{align}
If $f \in X_{d, \varepsilon}$, then the restriction of $f$ onto the hyperplane $H := \{(\theta_1, \theta_2, \theta_3) \, : \, \theta_1 + \theta_2 + \theta_3 = 0\}$ is supported in $\Lambda + B_\varepsilon(0)$, where $\Lambda = \mathbb{Z}v_1 \oplus \mathbb{Z}v_2$ and 
$$
    v_1 := (\frac{2\pi}{3}, -\frac{\pi}{3}, -\frac{\pi}{3}), \qquad  v_2 := ( -\frac{\pi}{3}, \frac{2\pi}{3}, -\frac{\pi}{3})\,.
$$
Furthermore, the function $\lvert f\rvert$ is periodic with respect to $\Lambda$, since it is periodic with respect to $\pi\mathbb{Z}^3$ and invariant under all translations in direction  $(1,1,1)$. Thus it suffices to show:

\begin{lemma}
    \label{ReducedLemma}
    Suppose that $\varepsilon \leq 1/20$. Then for all functions $f: H \to [0, \infty)$ that are periodic with respect to $\Lambda$ and supported in $\Lambda + B_\varepsilon(0)$, it holds that $I \geq I\!I$.
\end{lemma}

\begin{proof}
Note that $C$ is a fundamental domain of the lattice $2 \Lambda$. The expressions in the integrals for term $I$ and $I\!I$ are $2\Lambda$ periodic, so we may replace $C$ by any other fundamental domain $C'$. Since $f$ is supported in $\Lambda + B(0,\varepsilon)$, there exists a fundamental domain $C'$ such that $f|_{C'}$ is supported in 
$$
    B_{\varepsilon}(0,0,0) \cup B_{\varepsilon}(\frac{2\pi}{3}, -\frac{\pi}{3}, -\frac{\pi}{3}) \cup B_{\varepsilon}(-\frac{\pi}{3}, \frac{2\pi}{3}, -\frac{\pi}{3}) \cup B_{\varepsilon}(-\frac{\pi}{3}, -\frac{\pi}{3}, \frac{2\pi}{3}) =: B_1 \cup B_2 \cup B_3 \cup B_4\,. 
$$
We decompose
\begin{equation}
    \label{eq Ii def}
    I = \sum_{i = 1}^4 \int_{B_i} \lvert f(\theta)\rvert^2 (a(\theta) - 1) \int_C \delta(a(\theta) - a(\theta'))   \, \mathrm{d}\mathcal{H}^2_C(\theta')\,
    \mathrm{d}\mathcal{H}^2_C(\theta) =: \sum_{i = 1}^4 I_i\,,
\end{equation}
\begin{equation}
    \label{eq IIi def}
    I\!I = \sum_{1 \leq i,j \leq 4} \int_{B_i\times B_j}  \delta(a(\theta) - a(\theta')) \lvert a(\theta) - 1\rvert \lvert f(\theta)\rvert \lvert f(\theta')\rvert\, \mathrm{d}\mathcal{H}^2_C(\theta)\, \mathrm{d}\mathcal{H}^2_C(\theta') =: \sum_{1 \leq i,j \leq 4} I\!I_{ij}\,.
\end{equation}
Note that $\lvert \theta\rvert < \pi/6$ implies, by \eqref{eq a cos}, that $a(\theta) \geq 3 + 6 \cos(\pi/3) = 6$,
and $\lvert \theta - (2\pi/3, -\pi/3, -\pi/3)\rvert  < \pi/6$ implies similarly that $a(\theta) \leq 3$. Therefore, for $j = 2,3,4$ the measure $\delta(a(\theta) - a(\theta'))$ vanishes on $B_1 \times B_j$, thus $I_{1j} = I_{j1} = 0$.
Next, we record that $I\!I_{11}\leq I_1$, by Cauchy-Schwarz and since $a(\theta) \geq 6$ on $B_1$:
\begin{align*}
    I\!I_{11} &= \int_{B_1^2} \delta(a(\theta) -  a(\theta')) \lvert a(\theta) - 1\rvert \lvert f(\theta)\rvert\lvert f(\theta')\rvert \mathrm{d}\mathcal{H}^2_C(\theta) \,\mathrm{d}\mathcal{H}^2_C(\theta')\\
    &\leq \frac{1}{2} \int_{B_1^2} \delta(\tilde a(\theta) - \tilde a(\theta')) (\tilde a(\theta) - 1) (\lvert f(\theta)\rvert^2 + \lvert f(\theta')\rvert ^2) \, \mathrm{d}\mathcal{H}^2_C(\theta)\,\mathrm{d}\mathcal{H}^2_C(\theta')\\
    &\leq \int_{B_1} \lvert f(\theta)\rvert^2 (\tilde{a}(\theta) -1) \int_{C} \delta(\tilde{a}(\theta) - \tilde{a}(\theta'))  \, \mathrm{d}\mathcal{H}^2_C(\theta)\,\mathrm{d}\mathcal{H}^2_C(\theta') = I_1\,.
\end{align*}
The remaining terms are estimated in the next two sections. By \cref{lem I est} and \cref{lem II est}, we have
$$
    I_2 + I_3 + I_4 \geq  30 \int_{B_1} \lvert \theta\rvert^2 \lvert f(\theta)\rvert^2 \, \mathrm{d}\mathcal{H}^2_{H}(\theta) > 9 \frac{101}{100} \pi \int_{B_1} \lvert \theta\rvert ^2 \lvert f(\theta)\rvert^2 \, \mathrm{d}\mathcal{H}^2_{H}(\theta) \geq \sum_{2 \leq i, j \leq 4} I\!I_{ij}\,,
$$
which completes the proof.
\end{proof}

\section{Estimating term \texorpdfstring{$I$}{I}}

\begin{lemma}
    \label{lem I est}
    It holds that 
    \begin{equation}
    \label{eq lem I est}
        I_2+ I_3 + I_4= \int_{B_1} m(\theta) \lvert f(\theta)\rvert^2 \, \mathrm{d}\mathcal{H}^2_{H}(\theta)\,,
    \end{equation}
    where $I_j$ is defined in \eqref{eq Ii def}, and $m(\theta) \geq 30\lvert \theta \rvert^2$.
\end{lemma}

\begin{proof}
    By definition of the $I_j$, equation \eqref{eq lem I est} holds with
    $$
        m(\theta) = \sum_{j = 2}^4 (a(\theta + c_j) - 1) \int_C \delta(a(\theta + c_j) - a(\theta')) \, \mathrm{d}\mathcal{H}^2_C(\theta')\,,
    $$
    where $c_j$ is the center of the ball $B_j$. 
    Reversing the argument in the proof of  \cref{lem dimred}, it follows that for $x \in \R^2$
    \begin{multline*}
        \int_C \delta(\lvert x\rvert^2 - a(\theta')) \, \mathrm{d}\mathcal{H}^2_C(\theta') = \frac{1}{\sqrt{3}}\int_P \delta(\lvert x\rvert^2 - \lvert \omega_1 + \omega_2 + \omega_3\rvert^2) \delta(\arg(x) - \arg(\omega_1 + \omega_2 + \omega_3)) \, \prod_{j = 1}^3 \mathrm{d}\theta'_j\\
        = \frac{1}{2\sqrt{3}} \int_{(S^1)^3} \delta(x - (\omega_1 + \omega_2 + \omega_3))  \, \prod_{j = 1}^3 \mathrm{d}\sigma(\omega_j)= \frac{1}{2\sqrt{3}} \sigma * \sigma * \sigma(x)\,.
    \end{multline*}
    The convolution $\sigma*\sigma*\sigma$ is radial. We set $\sigma * \sigma * \sigma(x) = \rho(\lvert x\rvert)$, giving
    \begin{equation}
        \label{mdef}
        m(\theta) = \frac{1}{2\sqrt{3}} \sum_{j = 2}^4 (a(\theta + c_j) - 1) \rho(\sqrt{a(\theta + c_j)})\,.
    \end{equation}
    In polar coordinates 
    \begin{align}
        \label{variables1}
        \begin{pmatrix}
            \theta_1\\
            \theta_2\\
            \theta_3
        \end{pmatrix}
        &=
        s\cos(\alpha) \frac{1}{\sqrt{2}}
        \begin{pmatrix}
            1\\
            -1\\
            0
        \end{pmatrix}
        + s\sin(\alpha) \frac{1}{\sqrt{6}}
        \begin{pmatrix}
            1\\
            1\\
            -2
        \end{pmatrix}
    \end{align}
    we compute in \cref{auxilliary3} the asymptotic expansion
    \begin{align}
        (a(\theta + c_4) - 1)\rho(\sqrt{a(\theta + c_4)})
        &= -12s^2 (3\sin^2 (\alpha) - \cos^2(\alpha)) \log(s) \label{eq rho1}\\
        &\quad- 6 s^2 (3\sin^2 (\alpha) - \cos^2(\alpha)) \log \lvert 3\sin^2 (\alpha) - \cos^2(\alpha)\rvert \label{eq rho2}\\
        &\quad+ 18 \log 2 \, s^2 (3 \sin^2(\alpha) - \cos^2(\alpha)) \label{eq rho3}\\
        &\quad+ E\,, \nonumber
    \end{align}
    with $\lvert E\rvert \leq - 180 s^4 \log s + 71 s^4$ when $s \leq 1/20$. As the function $a$ is invariant under permutation of its arguments and constant in direction $(1,1,1)$, it is invariant under the rotation $T$ by $2\pi/3$ about the line $\R (1,1,1)$. Since 
    $$
        c_2 + \theta(\alpha, s) = T(c_4 + \theta(\alpha + 4\pi/3, s))\,, \qquad c_3 + \theta(\alpha, s) = T^2(c_4 + \theta(\alpha + 2\pi/3, s))\,,
    $$
    we obtain the same asymptotic expansion for $a(\theta + c_j) \rho(\sqrt{a(\theta + c_j)})$, $j = 2,3$, but with $\alpha$ replaced by $\alpha + \frac{4\pi}{3}$ and $\alpha + \frac{2\pi}{3}$. 

    We now consider \eqref{mdef}. The term \eqref{eq rho1} contributes $-6 \sqrt{3} s^2 \log(s)$ to $m$ and the term \eqref{eq rho3} contributes $9 \sqrt{3} \log(2) s^2$, since for all $\alpha$
    $$
    \sum_{j = 1}^3 (3\sin^2(\alpha + \frac{2\pi j}{3}) - \cos^2(\alpha + \frac{2\pi j}{3}))  =  3\,.
    $$
    For term \eqref{eq rho2} we use the sharp estimate
    $$
        \sum_{j = 1}^3 (3\sin^2(\alpha + \frac{2\pi j}{3}) - \cos^2(\alpha +  \frac{2\pi j}{3})) \log \lvert 3\sin^2(\alpha + \frac{2\pi j}{3}) - \cos^2(\alpha + \frac{2\pi j}{3})\rvert \leq 3 \log(3)\,,
    $$
    which we prove in \cref{auxilliary4}.
    Hence, for $s \leq 1/20$,
    \begin{align*}
        m(\theta) &\geq  -6\sqrt{3}s^2 \log(s) + (9\sqrt{3}\log(2) - 3\sqrt{3}\log(3)) s^2 + 90\sqrt{3} s^4\log s - 62 s^4 \\
        &\geq (6 \sqrt{3} \log(20) + 9 \sqrt{3} \log(2) - 3\sqrt{3}\log(3) - 90\sqrt{3}/400 \log(20) - 62/400) s^2 \approx 34.906 \ s^2\,,
    \end{align*}
    as claimed.
\end{proof}

\section{Estimating term \texorpdfstring{$I\!I$}{II}}

\begin{lemma}
    \label{lem II est}
    For all $2 \leq i,j \leq 4$ and all $f$, it holds that 
    $$
        I\!I_{ij} \leq \frac{101}{100}\pi \int_{B_1} \lvert \theta\rvert ^2 \lvert f(\theta)\rvert^2 \, \mathrm{d}\mathcal{H}^2_H(\theta)\,.
    $$
\end{lemma}

\begin{proof}
    We first treat the term $I\!I_{44}$, and later explain the changes for the other terms.
    We have:
    \begin{align*}
        I\!I_{44} &= \int_{B_4 \times B_4} \delta(1 - \frac{1 -  a(\theta')}{1 - a(\theta)}) \lvert f(\theta)\rvert\lvert f(\theta')\rvert \, \mathrm{d}\mathcal{H}^2_H(\theta) \, \mathrm{d}\mathcal{H}^2_H(\theta')\\
        &= \int_{B_1 \times B_1} \delta(1 - \frac{1 -  a(c_4 + \theta')}{1 - a(c_4 + \theta)}) \lvert f(\theta)\rvert \lvert f(\theta')\rvert \, \mathrm{d}\mathcal{H}^2_H(\theta) \, \mathrm{d}\mathcal{H}^2_H(\theta')
    \end{align*} 
    We introduce polar coordinates $\theta = \theta(s, \alpha)$ as in \eqref{variables1} and write also $\theta' = \theta(t, \beta)$. With the definitions $h(s,t,\alpha, \beta) := (1 - a(c_4 + \theta'))/(1 - a(c_4 + \theta))$ and  $g(s, \alpha) := \lvert \theta\rvert^2\lvert f(\theta)\rvert$, we obtain by changing variables
    \begin{equation}
        \label{eq I44 rewritten}
        I\!I_{44} = \int_{0}^{2\pi} \int_{0}^{2\pi} \int_0^\varepsilon \int_0^\varepsilon \delta(1 - h(s,t,\alpha, \beta)) g(s,\alpha)g(t,\beta) \,\frac{\mathrm{d}s}{s} \, \frac{\mathrm{d}t}{t} \, \mathrm{d}\alpha \, \mathrm{d}\beta\,.
    \end{equation}
    Doing a Taylor expansion of $1 - a(c_4 + \theta)$ at $0$ yields (see \cref{Auxilliary2})
    \begin{equation}
        \label{hdef}
        h(s,t,\alpha, \beta) = \frac{t^2}{s^2} \frac{3 \sin^2(\beta) - \cos^2(\beta)}{3\sin^2(\alpha) - \cos^2(\alpha)} \frac{1 + \psi(t, \beta)}{1 + \psi(s, \alpha)}\,,
    \end{equation}
    where $\psi$ is a smooth function of $s, \alpha$, and $\psi(s, \alpha)= O(s^2)$.
    If the last factor in \eqref{hdef} was equal to $1$, then the inner two integrals in \eqref{eq I44 rewritten} would simplify to 
    $$
        \int_0^\infty g(s, \alpha)g(c(\alpha, \beta) s, \beta) \,\frac{\mathrm{d}s}{s} \,,
    $$
    for some constant $c(\alpha, \beta)$,
    which is easily estimated using Cauchy-Schwarz. The following is a perturbed version of this argument.

    Fix $\alpha$, $\beta$ and write $h(s,t) = h(s,t,\alpha, \beta)$.
    Let $s(t)$ be defined implicitly by $h(s(t), t) = 1$ (note that $s$ also depends on $\alpha$ and $\beta$). Then
    \begin{align}
        \int_0^\varepsilon \int_0^\varepsilon \delta(1 - h(s,t))  g(s,\alpha)g(t,\beta) \,\frac{\mathrm{d}s}{s}\,\frac{\mathrm{d}t}{t}&= \int_0^\varepsilon g(t,\beta)g(s(t),\alpha) \frac{1}{\lvert \partial_s h(s(t),t)\rvert } \frac{1}{s(t)t} \, \mathrm{d}t \nonumber\\
        &= \int_0^\varepsilon g(t,\beta)g(s(t),\alpha) \frac{1}{2 + s(t)\frac{\psi'(s(t),\alpha)}{1 + \psi(s(t),\alpha)}} \frac{1}{t} \, \mathrm{d}t\,. \label{eq g}
    \end{align}
    Here we used that 
    $$
        \partial_s h(s,t) = -h(s,t)(\frac{2}{s} + \frac{\psi'(s,\alpha)}{1 + \psi(s,\alpha)})
    $$
    and hence
    $$ -\partial_s h(s(t),t) = \frac{2}{s(t)} + \frac{\psi'(s(t),\alpha)}{1 + \psi(s(t),\alpha)}\,. $$
    Applying Cauchy-Schwarz, we obtain 
    \begin{equation}
        \label{eq GM integrals}
        \eqref{eq g} \leq \left(\int_0^\varepsilon g(t,\beta)^2 \frac{1}{2 + s(t)\frac{\psi'(s(t),\alpha)}{1 +\psi(s(t),\alpha)}} \frac{1}{t} \, \mathrm{d}t \right)^{1/2} \left(\int_0^\varepsilon  g(s(t),\alpha)^2 \frac{1}{2 + s(t)\frac{\psi'(s(t),\alpha)}{1 + \psi(s(t),\alpha)}} \frac{1}{t} \, \mathrm{d}t\right)^{1/2}\,.
    \end{equation}
    After substituting $s = s(t)$ in the second integral, its integrand becomes the same as in the first one, but with the roles of $(s, \alpha)$ and $(t,\beta)$ interchanged. 
    By \cref{Auxilliary2}, it holds for $s \leq 1/20$ that 
    $$
        \lvert \psi(s, \alpha)\rvert < 1/100 \quad \text{and} \quad \lvert \psi'(s, \alpha)\rvert \leq 1/10\,,
    $$
    giving
    $$
        \lvert s \frac{\psi'(s, \alpha)}{1 + \psi(s, \alpha)}\rvert \leq \frac{1}{198}\,.
    $$
    Thus, the factor in the integrals in \eqref{eq GM integrals} is bounded above by $198/395 < 101/200$. It follows that 
    \begin{multline*}
        I\!I_{44} \leq \frac{101}{200} \int_0^{2\pi} \int_0^{2\pi} \left(\int_0^\varepsilon g(t,\beta)^2 \frac{\mathrm{d}t}{t} \right)^{1/2} \left(\int_0^\varepsilon  g(s,\alpha)^2  \frac{\mathrm{d}s}{s} \right)^{1/2} \, \mathrm{d}\alpha \, \mathrm{d}\beta\\
        \leq \frac{101}{100}\pi \int_0^{2\pi} \int_0^\varepsilon \lvert g(s, \alpha)\rvert^2 \, \frac{\mathrm{d}t}{t} \, \mathrm{d}\alpha
        = \frac{101}{100}\pi \int_{B_1} \lvert \theta\rvert^2 \lvert f(\theta)\rvert^2 \, \mathrm{d} \mathcal{H}^2_H(\theta)\,.
    \end{multline*}
    For the other eight integrals the same estimate holds: By the argument in the proof of \cref{lem I est}, changing $c_4$ to some other $c_j$ only changes the expansion in \eqref{hdef} by a translation in $\alpha$ and $\beta$. Then the rest of the argument goes through exactly as for $I\!I_{44}$.
\end{proof}

\section{Technical estimates}
Here we prove the computational lemmas that were used in the main argument. 

We have the following explicit formula for $\rho$ (see \cite{Carneiro+2017}, Lemma 8):
\begin{equation}
    \label{rhoformula}
    \rho(r) = \frac{4}{r} \int_{A(r)}^1 \frac{\mathrm{d}u}{\sqrt{1 - u^2}\sqrt{\frac{(1-r)^2}{2r} + 1 - u}\sqrt{\frac{(3+r)(1-r)}{2r} + 1 + u}}
\end{equation}
with $A(r) = -1 + \max\{0, (3+r)(r-1)/(2r)\}$. From this we obtain the following asymptotic formula:

\begin{lemma}
\label{AsymptoticsRho}
Let $\rho$ be defined by $\rho(\lvert x\rvert) = \sigma* \sigma * \sigma(x)$. Then we have for all $r$ with $\lvert r - 1\rvert \leq 1/10$
\begin{align*}
    \lvert\rho(r) + 6 \log \lvert 1 - r\rvert - 12 \log 2\rvert \leq -22 \lvert r - 1\rvert \log \lvert r - 1\rvert + 23 \lvert r-1\rvert\,.
\end{align*}
\end{lemma}
We have not tried to optimize the error in this estimate. We give an elementary, self-contained proof below. For an alternative proof one can use the identity (see  \cite{Pearson06} or \cite{Borwein+12})
\begin{equation*}
    \label{eq Pearson}
    \rho(x) = 
    \begin{cases}
        \frac{16}{\sqrt{(x+1)^3(3-x)}}K\left( \sqrt{\frac{16x}{(x+1)^3(3-x)}}\right) \quad &\text{if $0 \leq x < 1$}\\
        \frac{4}{\sqrt{x}} K\left( \sqrt{\frac{(x+1)^3(3-x)}{16x}}\right) \quad &\text{if $1 < x \leq 3$}\\
        0 \quad &\text{if $x > 3$}
    \end{cases}
\end{equation*}
where 
$$
    K(k) = \int_0^1 \frac{1}{\sqrt{1 - x^2}\sqrt{1 - k^2x^2}} \, \mathrm{d}x
$$
is the complete elliptic integral of the first kind, together
with known asymptotics for $K(k)$ as $k \nearrow 1$. Since we could not find a proof of \eqref{eq Pearson} in the literature, and its proof is indeed not short, we have decided to give the elementary argument instead. 

We first prove some auxilliary lemmas.

\begin{lemma}
    \label{Auxilliary1}
    For all $\delta > 0$ it holds that 
    $$
        0 \leq \int_0^1 \frac{1}{\sqrt{u}\sqrt{u + \delta}} \, \mathrm{d}u - \log(\frac{4}{\delta}) \leq \frac{1}{2} \delta\,.
    $$
\end{lemma}

\begin{proof}
    We have 
    $$
        \int_0^1 \frac{1}{\sqrt{u}\sqrt{u + \delta}} \, \mathrm{d}u  = -\log(\delta) + 2 \log(1 + \sqrt{1 + \delta})\,.
    $$
    Furthermore, by the mean value theorem, there exists $0 < \delta' < \delta$ such that
    \begin{align*}
        \log(1 + \sqrt{1 + \delta}) = \log(2) + \delta g(\delta')
    \end{align*}
    where 
    $$
        0 < g(\delta) = \frac{1}{2(1 + \sqrt{1 + \delta})\sqrt{1 + \delta}} \leq \frac{1}{4}
    $$
    is the derivative of $\log(1 + \sqrt{1 + \delta})$. 
\end{proof}

\begin{lemma}
    \label{auxilliary5}
    For all $0 < a,b < 1$, we have:
    $$
        \left\lvert \int_0^1 \frac{1}{\sqrt{1 - x^2}\sqrt{a + 1 - x}\sqrt{b + 1 +x}} \, \mathrm{d}x - \int_0^1 \frac{1}{\sqrt{1 - x^2}\sqrt{a + 1 - x}\sqrt{1 +x}} \, \mathrm{d}x \right\rvert \leq \frac{b}{2}(\log(\frac{4}{a}) + \frac{a}{2})\,.
    $$
\end{lemma}

\begin{proof}
    By the mean value theorem, we have for all $x \geq 0$
    $$
        \lvert (b + 1 + x)^{-1/2} - (1 + x)^{-1/2}\rvert \leq \frac{1}{2}b\,.
    $$
    Hence the left hand side of the claimed inequality is estimated by 
    $$
        \frac{b}{2}\int_0^1 \frac{1}{\sqrt{1-x}\sqrt{a + 1 - x}} \, \mathrm{d}x \leq \frac{b}{2}(\log(\frac{4}{a}) + \frac{a}{2})\,,
    $$
    where we applied \cref{Auxilliary1}.
\end{proof}

\begin{lemma}
    \label{auxilliary6}
    For all $1 > a > 0$, we have 
    $$
        \left\lvert \int_0^1 \frac{1}{(1 +x)\sqrt{1 -x}\sqrt{a + 1 - x}} \, \mathrm{d}x - \frac{1}{2} \log(\frac{8}{a})\right\rvert \leq \frac{1}{2}a \log(1 + \frac{1}{a})\,.
    $$
\end{lemma}

\begin{proof}
    We have with $v = 1-x$
    $$
        \int_0^1 \frac{1}{(1 +x)\sqrt{1 -x}\sqrt{a + 1 - x}} \, \mathrm{d}x = \int_0^1 \frac{1}{(2 - v)\sqrt{v}\sqrt{a + v}} \, \mathrm{d}v
    $$
    which can be expanded to equal
    $$
        \frac{1}{2}\int_0^1 \frac{1}{\sqrt{v}\sqrt{a + v}} \, \mathrm{d}v + \frac{1}{2}\int_0^1 \frac{1}{2 - v} \, \mathrm{d}v - \frac{a}{2} \int_0^1 \frac{1}{(2 - v)\sqrt{a + v}(\sqrt{v} + \sqrt{a + v})} \, \mathrm{d}v\,.
    $$
    Computing the second integral and using \cref{Auxilliary1} for the first one yields the main term $\log(8/a)/2$. For the error estimate we combine \cref{Auxilliary1} and the bound
    $$
        \int_0^1 \frac{1}{(2 - v)\sqrt{a + v}(\sqrt{v} + \sqrt{a + v})} \, \mathrm{d}v \leq \int_0^1 \frac{1}{v + a}\,\mathrm{d}v =  \log(1 + \frac{1}{a}) \,,
    $$
    and note that the errors have opposite signs.
\end{proof}

\begin{proof}[Proof of \cref{AsymptoticsRho}]
    We start with the case $r = 1 - \varepsilon < 1$. By \eqref{rhoformula}, we have
    $$
        \frac{1 - \varepsilon}{4} \rho(1 - \varepsilon) = 
        \int_{-1}^1 \frac{1}{\sqrt{1 - u^2}\sqrt{\frac{\varepsilon^2}{2 - 2\varepsilon} + 1 - u}\sqrt{\frac{(4-\varepsilon)\varepsilon}{2 - 2\varepsilon} + 1 + u}} \, \mathrm{d}u\,.
    $$
    Combining \cref{auxilliary5} and \cref{auxilliary6} with $a = \varepsilon^2/(2 - 2\varepsilon)$ and $b = (4 - \varepsilon)\varepsilon/(2- 2\varepsilon)$, we obtain for this integral:
    $$
        = \frac{1}{2}(\log(\frac{8}{a}) + \log(\frac{8}{b})) + E = 3 \log(2)  - \frac{3}{2} \log(\varepsilon) - \log(2 - 2\varepsilon) + \frac{1}{2} \log(4 -\varepsilon) + E
    $$
    with 
    \begin{equation}
        \label{eq error estimate}
        \lvert E\rvert \leq \frac{1}{2}(b\log(\frac{4}{a}) + a\log(\frac{4}{b}) + ab + a \log(1 + \frac{1}{a}) + b \log(1 + \frac{1}{b}))\,.
    \end{equation}
    It is easy to see that 
    $$
        \lvert\frac{1}{2} \log(4 - \varepsilon) - \log(2 - 2\varepsilon)\rvert \leq \frac{\varepsilon}{2}\,.
    $$
    Further, one verifies that, when $0 < \varepsilon \leq 1/10$,
    $$
        a \leq \frac{1}{18}\varepsilon\,, \ b \leq \frac{19}{9} \varepsilon\,, \ \log(\frac{4}{a}) \leq 3 \log(2) - 2 \log(\varepsilon)\,, \ \log(\frac{4}{b}) \leq \log(2) - \log(\varepsilon)
    $$
    and 
    $$
        \log(1 + \frac{1}{a}) \leq \log(2) - 2\log(\varepsilon)\,, \ \log(1 + \frac{1}{b}) \leq -\log(\varepsilon)\,.
    $$
    Using this, one can check that
    $$
        E \leq \frac{13}{4} \varepsilon \log(\frac{1}{\varepsilon}) + \frac{5}{2} \varepsilon\,.
    $$
    To summarize, we have shown that 
    $$
        \left\lvert \frac{1 - \varepsilon}{4} \rho(1 - \varepsilon) - 3 \log(2) + \frac{3}{2} \log(\varepsilon) \right\rvert \leq \frac{13}{4} \varepsilon \log(\frac{1}{\varepsilon}) + 3 \varepsilon\,.
    $$
    We multiply by $4/(1 - \varepsilon)$, and use that $\lvert 4/(1 - \varepsilon) - 4\rvert \leq 40/9 \varepsilon$ to obtain
    $$
        \lvert \rho(1 - \varepsilon) - 12 \log(2) + 6 \log(\varepsilon)\rvert \leq 22 \varepsilon \log(\frac{1}{\varepsilon}) +  23\varepsilon\,.
    $$

    \noindent Now we turn to the case $r = 1 + \varepsilon > 1$. There we have
    $$
        \rho(1 + \varepsilon) = \frac{4}{1 + \varepsilon} \int_{-1 + \frac{(4 + \varepsilon)\epsilon}{2 + 2\varepsilon}}^1 \frac{1}{\sqrt{1 - u^2}\sqrt{\frac{\varepsilon^2}{2 + 2\varepsilon} + 1 - u}\sqrt{-\frac{(4+\varepsilon)\varepsilon}{2 + 2\varepsilon} + 1 + u}} \,\mathrm{d}u
    $$
    $$
        = \frac{16}{4 - \varepsilon^2} \int_{-1}^1 \frac{1}{\sqrt{1 - v^2}\sqrt{\frac{2\varepsilon^2}{4 - \varepsilon^2} + 1 - v}\sqrt{\frac{8\varepsilon}{4- \varepsilon^2} + 1 + v}} \, \mathrm{d}v\,.
    $$
    We first approximate the integral. We can argue as in the case $r < 1$, now with $a = 2\varepsilon^2/(4 - \varepsilon^2)$ and $b = 8\varepsilon/(4 - \varepsilon^2)$. The main term is easily seen to be the same as in the case $r < 1$, and the error is bounded by
    $$
        -\log(1 - \frac{\varepsilon^2}{4}) + E \leq \frac{\varepsilon}{40} + E\,,
    $$
    with $E$ satisfying \eqref{eq error estimate}. Now we have
    $$
        a \leq \frac{1}{15}\varepsilon\,, \ b \leq \frac{800}{399} \varepsilon\,, \ \log(\frac{4}{a}) \leq 3 \log(2) - 2\log(\varepsilon)\,, \ \log(\frac{4}{b}) \leq \log(2) - \log(\varepsilon)
    $$
    and
    $$
        \log(1 + \frac{1}{a}) \leq \log(\frac{201}{100}) - 2\log(\varepsilon)\,, \ \log(1 + \frac{1}{b}) \leq -\log(\varepsilon)\,.
    $$
    Using this, we obtain
    $$
        \lvert E\rvert + \frac{\varepsilon}{40} \leq  \frac{13}{4} \varepsilon \log(\frac{1}{\varepsilon}) + \frac{9}{4}\varepsilon\,.
    $$
    In other words, it holds that 
    $$
        \left\lvert \frac{4 - \varepsilon^2}{16} \rho(1 + \varepsilon) - 3\log(2) + \frac{3}{2} \log(\varepsilon)\right\rvert \leq \frac{13}{4} \varepsilon \log(\frac{1}{\varepsilon}) + \frac{9}{4}\varepsilon\,.
    $$
    We multiply by $16/(4 - \varepsilon^2)$ and use that $\lvert 16/(4 - \varepsilon^2) - 4\rvert \leq 40/399 \varepsilon$ to obtain
    $$
        \lvert \rho(1 + \varepsilon) - 12 \log(2) + 6 \log(\varepsilon)\rvert \leq 14 \varepsilon \log(\frac{1}{\varepsilon}) + 9\varepsilon\,.
    $$
    This completes the proof.
\end{proof}

\begin{lemma}
    \label{Auxilliary2}
    Let $\theta$ be given by \eqref{variables1}. Then it holds that 
    $$
        a(c_4 + \theta) - 1 = s^2(3\sin^2(\alpha) - \cos^2(\alpha)) (1 + \psi(s, \alpha))\,,
    $$
    where $\psi(s, \alpha)$ is a smooth function satisfying the following estimates:
    \begin{align*}
        \lvert \psi(s, \alpha)\rvert & \leq \frac{7}{24}s^2 + \frac{17}{720} s^4 +  s^6 e^{\sqrt{2}s}\\
        \lvert \psi'(s, \alpha)\rvert &\leq \frac{14}{24} s + \frac{17}{180} s^3 + 2 s^5 e^{\sqrt{2}s}\,.
    \end{align*}
\end{lemma}

\begin{proof}
    By the definition of $h$, the trigonometric identities and the Taylor expansion of $\cos$, we have 
    \begin{align}
        a(c_4 + \theta)-1 
        &= a( (0,0,\pi) + \theta) - 1\nonumber\\
        &= (\cos(\theta_1) + \cos(\theta_2) - \cos(\theta_3))^2 + (\sin(\theta_1) + \sin(\theta_2) - \sin(\theta_3))^2  -1\nonumber\\    
        &= 2 + 2\cos(\theta_1 - \theta_2) - 2 \cos(\theta_1 - \theta_3) - 2\cos(\theta_2 - \theta_3)\nonumber\\
        &= 2\sum_{k = 1}^\infty \frac{(-1)^k}{(2k)!} ((\theta_1 - \theta_2)^{2k} - (\theta_1 - \theta_3)^{2k} - (\theta_2 - \theta_3)^{2k})\label{PolDef}\\
        &=: \sum_{k = 1}^\infty s^{2k} \frac{(-1)^k}{(2k)!}P_{2k}(\sin(\alpha), \cos(\alpha))\,.\nonumber
    \end{align}
    It follows from \eqref{PolDef} that each $P_{2k}$ vanishes when $\theta_1 = \theta_3$ and when $\theta_2 = \theta_3$, which is equivalent to $\alpha = \pm \pi/6$, or to $\cos(\alpha) = \pm \sqrt{3} \sin(\alpha)$. Hence, the homogeneous polynomial $P_{2k}(X, Y)$ vanishes on the lines $\sqrt{3}X+ Y = 0$ and $\sqrt{3}X - Y = 0$. We conclude that for all $k$, the factor $3X^2 - Y^2$ divides $P_{2k}(X,Y)$.
    Define $Q_{2k}$ by 
    $$
        Q_{2k}(X,Y)(3X^2 - Y^2) = (-1)^kP_{2k}(X,Y)\,.
    $$
    Then we have, using that $Q_{2} = 1$:
    \begin{align*}
        a(c_4 + \theta) - 1 = s^2 (3\sin^2(\alpha) - \cos^2(\alpha)) (1 + \psi(s, \alpha))
    \end{align*}
    where $\psi$ is defined by
    $$
        \psi(s, \alpha) = \sum_{k = 2}^\infty s^{2k-2} \frac{1}{(2k)!} Q_{2k}(\cos(\alpha), \sin(\alpha))\,.
    $$    

    Now we fix $k$ and estimate $p(\alpha) := P_{2k}(\sin(\alpha), \cos(\alpha))$ and $q(\alpha) := Q_{2k}(\sin(\alpha), \cos(\alpha))$. By \eqref{variables1}, we have that 
    $$
        \theta_1 - \theta_2 = \sqrt{2} \cos(\alpha)
    $$
    $$
        \theta_3 - \theta_1 = -\frac{1}{\sqrt{2}} \cos(\alpha) - \frac{\sqrt{3}}{\sqrt{2}} \sin(\alpha) = \sqrt{2} \cos(\alpha + \frac{2\pi}{3})
    $$
    $$
        \theta_2 - \theta_3 = -\frac{1}{\sqrt{2}} \cos(\alpha) + \frac{\sqrt{3}}{\sqrt{2}} \sin(\alpha) = \sqrt{2} \cos(\alpha -\frac{2\pi}{3})\,.
    $$
    Thus, by \eqref{PolDef},
    $$
        p(\alpha) = 2^{k+1} (-1)^k (\cos(\alpha)^{2k} - \cos(\alpha + \frac{2\pi}{3})^{2k} - \cos(\alpha - \frac{2\pi}{3})^{2k})\,,
    $$
    Taking derivatives, and noting that the terms inside the brackets are each at most $1$, we obtain:
    $$
        \lvert p(\alpha)\rvert  \leq 6 \cdot 2^k\,, \qquad
        \lvert p'(\alpha)\rvert \leq 12 k  2^k\,, \qquad
        \lvert p''(\alpha)\rvert \leq 24 k^2  2^k \,.
    $$
    Denote 
    $$
        q(\alpha) := \frac{p(\alpha)}{3\sin^2(\alpha) - \cos^2(\alpha)} = \frac{p(\alpha)}{(\sqrt{3}\sin(\alpha) - \cos(\alpha))(\sqrt{3}\sin(\alpha) + \cos(\alpha))}\,.
    $$
    If both factors $\lvert \sqrt{3}\sin(\alpha) \pm \cos(\alpha)\rvert$ are at least $1/2$, we have that 
    $$
        q(\alpha) \leq 24 \cdot 2^k \,.
    $$
    If not, then $\lvert \alpha - \pi/6\rvert < 1/5$ or $\lvert \alpha + \pi/6\rvert < 1/5$. Without loss of generality we are in the first case. Then, by Taylor's formula:
    $$
        \left\lvert \frac{p(\alpha)}{\alpha - \pi/6} - p'(\pi/6)\right\rvert \leq \frac{1}{2}|\alpha - \pi/6| \sup |p''| \leq \frac{1}{10} 24 k^2 2^k \,,
    $$
    hence 
    $$
        \left\lvert  \frac{p(\alpha)}{\alpha - \pi/6} \right\rvert \leq 15 k^2 2^k\,.
    $$
    Furthermore, since $\lvert \alpha - \pi/6\rvert \leq 1/5$,
    $$
        \left\lvert \frac{\alpha - \pi/6}{(\sqrt{3}\sin(\alpha) - \cos(\alpha))(\sqrt{3}\sin(\alpha) + \cos(\alpha))}\right\rvert \leq 2 \left\lvert \frac{\alpha - \pi/6}{\sqrt{3}\sin(\alpha) - \cos(\alpha)}\right\rvert \leq \frac{1/5}{\sin(1/5)} < 2\,.
    $$
    \begin{table}[t]
    \centering
    \begin{tabular}{|c|c | c |} 
     \hline
     $k$ & $P_{2k}(X,Y)$ & $Q_{2k}(X,Y)$ \\ 
     \hline
     $1$ & $-3X^2 + Y^2$ & $1$\\
     \hline
     $2$ & $-9X^4 - 18X^2 Y^2 + 7 Y^4$ & $-3X^2 - 7Y^2$\\
     \hline
     $3$ & $-\frac{1}{2}(27X^6 + 135X^4Y^2 + 45X^2Y^4 - 31Y^6)$ & $\frac{1}{2}(9X^4 + 48X^2Y^2 + 31Y^4)$\\
     \hline
    \end{tabular}
    \caption{The polynomials $P_{2k}$, $Q_{2k}$ for small values of $k$.}
    \label{table:1}
    \end{table}
    Multiplying the last two estimates, we conclude that $\lvert q\rvert\leq 30 k^2 2^k$. We also directly compute for small $k$:
    $$
        \lvert Q_4(\sin(\alpha), \cos(\alpha))\rvert = \lvert - 7\cos^2(\alpha) - 3\sin^2(\alpha)\rvert \leq 7
    $$
    and 
    $$
        \lvert Q_6(\sin(\alpha), \cos(\alpha))\rvert  = \frac{1}{2}\lvert 9 \sin^4(\alpha) + 48 \sin^2(\alpha)\cos^2(\alpha) + 31 \cos^2(\alpha)\rvert \leq \frac{5125}{312} < 17\,.
    $$
    Plugging in these estimates, we obtain
    $$
        \lvert \psi(s, \alpha)\rvert \leq \frac{7}{24}s^2 + \frac{17}{720} s^4 + \sum_{k = 4}^\infty \frac{60 k^2}{(2k)!}(\sqrt{2} s)^{2k-2} \leq \frac{7}{24}s^2 + \frac{17}{720} s^4 +  s^6 e^{\sqrt{2}s}\,
    $$
    and
    $$
        \lvert \psi'(s, \alpha)\rvert \leq \frac{14}{24} s + \frac{17}{180} s^3 + \sqrt{2}\sum_{k = 4}^\infty \frac{60k^2}{(2k - 1)!}(\sqrt{2}s)^{2k-3} \leq \frac{14}{24} s + \frac{17}{180} s^3 + 2 s^5 e^{\sqrt{2}s}\,,
    $$
    as claimed.
\end{proof}

\begin{lemma}
    \label{auxilliary3}
    Let $\theta$ be given by \eqref{variables1}. Then  for all $0 \leq s \leq 1/20$, we have
    \begin{align*}
        (a(c_4 + \theta) - 1)\rho(\sqrt{a(c_4 + \theta)})
        &=
        -12s^2 (3\sin^2 (\alpha) - \cos^2(\alpha)) \log(s) \\
        &\quad- 6 s^2 (3\sin^2 (\alpha) - \cos^2(\alpha)) \log \lvert 3\sin^2 (\alpha) - \cos^2(\alpha)\rvert\\
        &\quad+ 18 \log 2 \, s^2 (3 \sin^2(\alpha) - \cos^2(\alpha))+ E\,,
    \end{align*}
    with 
    $$
        \lvert E \rvert \leq - 180 s^4 \log s + 71 s^4\,.
    $$
\end{lemma}

\begin{proof}
    By \cref{AsymptoticsRho}, it holds that
    \begin{align}
        (x^2 - 1)\rho(x)
        &= -6(x^2 - 1)\log\lvert x-1\rvert + 12 \log (2) (x^2 - 1) + (x^2 - 1)E_1\nonumber\\
        &= -6(x^2 - 1)\log\lvert x^2 - 1\rvert + 18 \log(2) (x^2 - 1) + (x^2 - 1)E_1 + 6(x^2 - 1) \log(1+ \frac{1}{2}(x - 1))\,,\label{eq1aux3}
    \end{align}
    where
    $$
        \lvert E_1 \rvert \leq -22 \lvert x - 1\rvert \log\lvert x-1\rvert + 23\lvert x-1\rvert \,. 
    $$
    Denote also the last term in \eqref{eq1aux3} by $E_2$.
    We set $x = \sqrt{a(c_4 + \theta)}$. \Cref{Auxilliary2} implies that $\lvert x-1\rvert \leq \lvert x^2 - 1\rvert  \leq 2 s^2$. Using this and monotonicity of $r \log r$, we obtain
    \begin{equation}
        \label{eq e1}
        \lvert(x^2 - 1)E_1\rvert \leq \lvert x^2 - 1\rvert(-22 \lvert x - 1\rvert \log \lvert x - 1\rvert + 23 \lvert x - 1\rvert ) \leq -176 s^4 \log(s)+ 32 s^4
    \end{equation}
    and
    \begin{equation}
        \label{eq e2}
        \lvert E_2 \rvert \leq 6\lvert x^2 - 1\rvert  \lvert \log(1+ \frac{1}{2}(x - 1))\rvert \leq 24 s^4\,.
    \end{equation}
    By \cref{Auxilliary2}, it holds that
    \begin{align}
        &\quad-6(x^2 -1) \log\lvert x^2 - 1\rvert\nonumber\\
        &= -6 s^2 (3 \sin^2(\alpha) - \cos^2(\alpha)) (1 + \psi(s, \alpha)) ( 2 \log(s) + \log(3\sin^2(\alpha) - \cos^2(\alpha)) + \log(1 + \psi(s,\alpha)))\nonumber\\
        &= -12 s^2 \log(s) (3 \sin^2(\alpha) - \cos^2(\alpha)) - 6s^2 (3 \sin^2(\alpha) - \cos^2(\alpha))  \log\lvert 3 \sin^2(\alpha) - \cos^2(\alpha)\rvert\label{eq m1}\\
        &\quad - 6 s^2 (3 \sin^2(\alpha) - \cos^2(\alpha)) \log(1 + \psi(s,\alpha)))\label{e2aux3}\\
        &\quad -6 s^2 \psi(s, \alpha) (3 \sin^2(\alpha) - \cos^2(\alpha)) ( 2 \log(s) + \log\lvert3\sin^2(\alpha) - \cos^2(\alpha)\rvert + \log(1 + \psi(s,\alpha)))\,. \label{e3aux3}
    \end{align}
    The term \eqref{e2aux3} bounded by $18 s^2 \lvert\psi(s,\alpha)\rvert \leq 6s^4$.
    The term \eqref{e3aux3} is bounded by 
    $$
        -18 s^2 \log(s) \lvert\psi(s, \alpha)\rvert + 4 s^2 \lvert \psi(s,\alpha)\rvert + 18 s^2 \psi(s, \alpha)^2 \leq -6 s^4 \log(s) + 2 s^4\,.
    $$
    For the second term in \eqref{eq1aux3}, we have
    \begin{equation}
        \label{eq m2}
        18 \log(2) (x^2 - 1)= 18 \log(2) s^2 (3\sin^2(\alpha) - \cos^2(\alpha))  + 18 \log(2) s^2 (3\sin^2(\alpha) - \cos^2(\alpha)) \psi(s, \alpha)\,,
    \end{equation}
    with the second term bounded by $27 \log(2) s^2 \lvert \psi(s, \alpha)\rvert \leq 9 \log(2) s^4$.
    Putting together the main terms \eqref{eq m1} and \eqref{eq m2}, and the estimates for the error terms \eqref{eq e1}, \eqref{eq e2}, \eqref{e2aux3}, \eqref{e3aux3} and in \eqref{eq m2}, one obtains the lemma.
\end{proof}

\begin{lemma}
    \label{auxilliary4}
    For all $\alpha$, it holds that 
    $$
        \sum_{j = 1}^3 (3\sin^2(\alpha +  \frac{2\pi j}{3}) - \cos^2(\alpha +  \frac{2\pi j}{3})) \log \lvert 3\sin^2(\alpha + \frac{2\pi j}{3}) - \cos^2(\alpha + \frac{2\pi j}{3})\rvert \leq 3 \log(3)\,.
    $$
\end{lemma}

\begin{proof}
    Let
    $$
        a_j = \sin^2(\alpha + \frac{2\pi j}{3}) - \frac{1}{3}\cos^2(\alpha + \frac{2\pi j}{3}) = \frac{1}{3} - \frac{2}{3} \cos(2\alpha +  \frac{4\pi j}{3})\,.
    $$
    It is easy to check that
    \begin{equation}
        \label{aproperties}
        a_1 + a_2 + a_3 = 1\,, \qquad
        a_1^2 + a_2^2 + a_3^2 = 1\,.
    \end{equation}
    Defining $b_j = (a_j + a_{j-1})/2$ (note that $a_{j+3} = a_j$), it follows that 
    $$
        b_1 + b_2 + b_3 = 1\,, \qquad b_1^2 + b_2^2 + b_3^2 = 1/2\,,
    $$
    hence $b_1, b_2, b_3 \geq 0$. Using Jensen's inequality, we deduce 
    $$
        \sum_{j = 1}^3 a_j \log(\lvert a_j\rvert) = \sum_{j = 1}^3 b_j \log(\frac{\lvert a_j\rvert \lvert a_{j-1}\rvert }{\lvert a_{j-2}\rvert }) \leq \log(\sum_{j=1}^3 b_j\frac{\lvert a_j\rvert \lvert a_{j-1}\rvert}{\lvert a_{j-2}\rvert})\,.
    $$
    By \eqref{aproperties}, we have that
    $$
        2 a_j a_{j-1} = (a_j + a_{j-1})^2 - (a_j^2 + a_{j-1}^2) = (1 - a_{j-2})^2 - (1 - a_{j-2}^2) = 2 a_{j-2} (a_{j-2} - 1)\,.
    $$
    Thus, using again \eqref{aproperties} 
    $$
        \sum_{j=1}^3 b_j\frac{\lvert a_j\rvert \lvert a_{j-1}\rvert }{\lvert a_{j-2}\rvert } = \sum_{j = 1}^3 b_j (1 - a_{j - 2}) = 1\,.
    $$
    We conclude that
    $$
        \sum_{j =1}^3 3a_j \log(\lvert 3 a_j\rvert ) = 3 \log(3) + 3 \sum_{j = 1}^3 a_j \log\lvert a_j\rvert \leq 3 \log 3\,. 
    $$
\end{proof}

\section{Discussion}
\label{sec discussion}
\subsection{Optimal value of \texorpdfstring{$\varepsilon$}{epsilon}}
An inspection of the above argument shows that $Q(f) \geq 0$ for all $f \in V_\varepsilon$ as long as
\begin{equation}
    \label{eq num}
    \inf_{\theta \in H, \lvert \theta\rvert \leq \varepsilon} \frac{1}{2}\sum_{j = 2}^4 (a(\theta + c_j) - 1)\rho(\sqrt{a(\theta + c_j)}) \geq 18 \pi \sup_{s \leq \varepsilon, \alpha \in [0, 2\pi]} \frac{1}{2 + s\frac{\psi'(s,\alpha)}{1 + \psi(s, \alpha)}}\,.
\end{equation}
(Non-rigorous) numerical computations suggest that this inequality holds up to $\varepsilon = 0.104$. The constant $\varepsilon'$ in \cref{cor main} could then be increased to $0.063$.
\begin{figure}[h]
\centering
\includegraphics[width=0.5\textwidth]{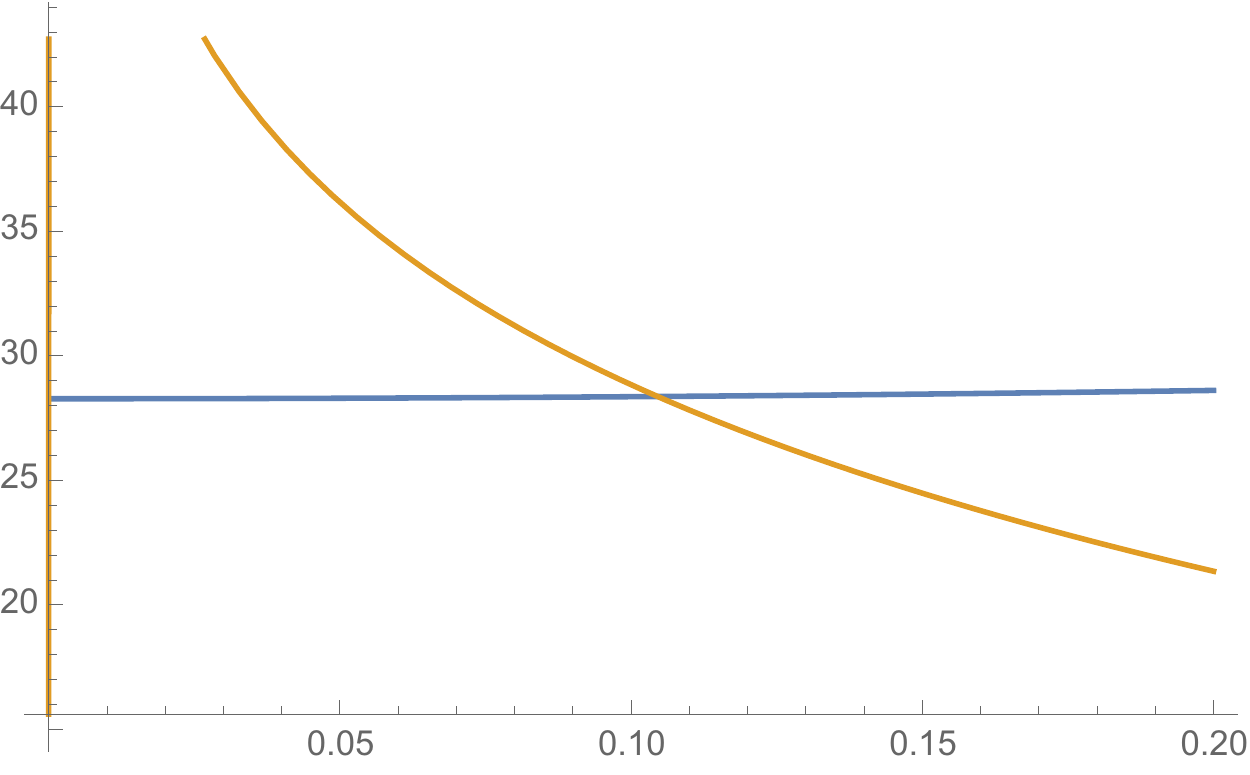}
\caption{The left hand side (yellow) and the right hand side (blue) of \eqref{eq num}.}
\end{figure}

\subsection{Fourier coefficients of \texorpdfstring{$Q$}{Q}}

In \cite{Barker+2020}, some numerical observations on the Fourier coefficients
$$
    \hat B(k,l) := B(\omega_1^{k_1} \omega_2^{k_2} \omega_3^{k_3}, \omega_4^{l_1} \omega_5^{l_2} \omega_6^{l_3})
$$
of $B$ with $k_1 + k_2 + k_3 = l_1 + l_2 + l_3 = 0$ are discussed. Namely, they are very large only when $k$ is very close to $l$ and when $k_1^2 + k_2^2 + k_3^2 \approx l_1^2 + l_2^2 + l_3^2$. We can explain this using \cref{lem dimred} as follows.

By \cref{lem dimred} and since $\lambda_0 = 1$, for all $f \in X_0$
$$
    B(f,f) = \int_C m(\theta) \lvert f(\theta)\rvert^2 \, \mathrm{d}\mathcal{H}^2_C(\theta) + \int_{C^2} n(\theta) \delta(a(\theta) - a(\theta')) f(\theta) \overline{f(\theta')} \, \mathrm{d}\mathcal{H}^2_C(\theta) \, \mathrm{d}\mathcal{H}^2_C(\theta')
$$
for certain functions $m, n$. 

The first term is a multiplier, hence it acts on the Fourier side by convolution with a fixed bump function. This bump function decays at least like $\lvert k-l\rvert^{-3}$, because the third derivative of $m$ is still integrable. This explains the large coefficients when $k$ is close to $l$.

The Fourier coefficients of the second term are the Fourier coefficients of the measure $\mu := n(\theta) \delta(a(\theta) - a(\theta'))$ supported on the $3$-manifold 
$$
    M := \{(x,y) \in C^2 \, : \, a(x) = a(y)\}\subset \R^6\,.
$$
The measure $\mu$ has a smooth, bounded density with respect to the Hausdorff measure on this manifold, except in the critical points of $a$. The Fourier transform of the parts where the measure has a smooth, bounded density can be estimated using the method of stationary phase and are of lower order than the contribution of the critical points. 
To explain what happens at a critical point (where $\det D^2 a \neq 0$), we choose coordinates $x_1, x_2, y_1, y_2$ for $C^2$, such that the critical point of $a$ is at $0$. After a scaling in $a$ and a linear change of variables, either 
\begin{equation}
    \label{eq aTaylor}
    a(x) = x_1^2 + x_2^2 + O(\lvert x\rvert^3) \quad \text{or} \quad a(x) = x_1^2 - x_2^2 + O(\lvert x\rvert^3)\,.
\end{equation}
Thus, ignoring higher order terms,
$$
    \delta(a(x) - a(y)) \approx \delta(\lvert x\rvert^2 - \lvert y\rvert^2)
    \quad \text{or}\quad
    \delta(a(x) - a(y)) \approx \delta(x_1^2 - x_2^2 - y_1^2 - y_2^2)\,.
$$
The Fourier transforms of these measures can be explicitly computed, in fact, they are their own Fourier transform. 
Now, $a$ has one local maximum and two local minima, which together with the above discussion explain why $\hat B(k,l)$ is very large on the cone $\lvert k\rvert^2 = \lvert l\rvert^2$. The contribution of all other critical points is of smaller order, since the weight $n$ vanishes there.

This disussion can be turned into a rigorous proof that the Fourier coefficients of $\mu$ concentrate near the cone $\lvert k\rvert^2 = \lvert l\rvert^2$. However, we can only show that they concentrate in e.g.
$$
    \{(k,l) : \lvert\lvert k\rvert - \lvert l\rvert\rvert \leq C\lvert k\rvert^{1/2}\}\,,
$$
and not in an $O(1)$ neighbourhood of the cone, because of the higher order terms in \eqref{eq aTaylor}.

\printbibliography

\end{document}